\documentclass[12pt]{amsart}

\usepackage[theoremfont]{newtxtext}

\usepackage{graphicx} 

\usepackage[T1]{fontenc}
\usepackage[utf8]{inputenc}
\usepackage{amsmath}
\usepackage{amssymb}
\usepackage{color}
\usepackage{amsthm}
\usepackage[matrix,arrow,curve]{xy}
\usepackage{mathrsfs}
\usepackage{tikz}
\usepackage{tikz-cd}
\usetikzlibrary{matrix,arrows,decorations.pathmorphing}
\usepackage{wrapfig}
\usepackage{eucal}

\usepackage[bookmarks=true,bookmarksopen=false]{hyperref}

\numberwithin{equation}{section}

\theoremstyle{plain}
\newtheorem{theorem}{Theorem}[section]
\newtheorem{proposition}[theorem]{Proposition}
\newtheorem{lemma}[theorem]{Lemma}

\theoremstyle{definition}
\newtheorem{definition}[theorem]{Definition}

\theoremstyle{remark}
\newtheorem{remark}[theorem]{Remark}

\def\co{\colon\thinspace}

\title{Acyclicity in the group-completion theorem}
\author{Simon Gritschacher}
\address{Mathematical Institute,
Ludwig-Maximilians-University,
Theresienstra{\ss}e 39,
80333 Munich,
Germany}
\email{simon.gritschacher@math.lmu.de}

\begin{document}

\maketitle

\begin{abstract}
    We prove acyclicity of the comparison map $M_\infty\to \Omega BM$ in the group-completion theorem under the assumption that $\pi_0(M)$ is a right denominator set in $H_\ast(M)$.
\end{abstract}

\section{Introduction}

Let $M$ be a topological monoid and assume that $\pi_0(M)\subseteq H_0(M)$ is a right denominator set in the integral Pontryagin ring $H_\ast(M)$. The group-completion theorem of McDuff-Segal \cite{mcduffsegal} asserts that the canonical map $M\to \Omega BM$ induces an isomorphism of rings
\begin{equation} \label{eq:mcduffsegal}
H_\ast(M)[\pi_0(M)^{-1}] \xrightarrow{\cong} H_\ast(\Omega B M)\, ,
\end{equation}
where the left hand side is the ring of right fractions of $H_\ast(M)$ with respect to $\pi_0(M)$. Our goal is to prove the stronger variant asserted in \cite[Remark 2]{mcduffsegal}. Suppose there is a sequence $s_1,s_2,\dots \in \pi_0(M)$ such that for every $n\geq 0$ and every $s\in \pi_0(M)$ there exists a $k\geq 0$ such that $s$ is a left factor of the product $s_{n+1}\cdots s_{n+k}$ in $\pi_0(M)$. We call such a sequence a cofinal right stabilisation sequence in $\pi_0(M)$. For each $s_i$ pick a representative $m_i\in M$ and form the mapping telescope
\[
M_\infty:=\mathrm{hocolim}(M\xrightarrow{\cdot m_1}M\xrightarrow{\cdot m_2}M\xrightarrow{\cdot m_3} \cdots)\, .
\]
There is a map $f\co M_\infty\to \Omega B M$ which induces (\ref{eq:mcduffsegal}) on homology.

\begin{proposition} \label{prop:main}
Let $M$ be a topological monoid such that
\begin{enumerate}
\item[(1)] $\pi_0(M)$ is a right denominator set in $H_\ast(M)$;
\item[(2)] a cofinal right stabilisation sequence in $\pi_0(M)$ exists.
\end{enumerate}
Then the comparison map $f\co M_\infty\to \Omega B M$ is acyclic.
\end{proposition}

The result is claimed in \cite[Remark 2]{mcduffsegal}, but no full proof is given. Later, Randal-Williams \cite{oscar} proved this result under the hypothesis that $M$ is homotopy commutative. In this article we give a different proof which does not use this hypothesis. An argument with the Serre spectral sequence reduces the proof to showing that the fundamental group of each component of $M_\infty$ acts trivially on the homology of its universal abelian covering. Triviality of this action is proved by an Eckmann-Hilton argument, based on the idea that the universal abelian coverings of the components of $M$ can be assembled into a topological monoid whose multiplication is compatible with the action by Deck transformations on each of these coverings. This argument bears some similarity to Wagoner's proof of \cite[Proposition 3.1]{wagoner} which, however, concerns special monoids derived from so-called ``direct sum groups''.

To clarify the logical dependence of our argument on previous results, let us recall that McDuff and Segal proved the following criterion (cf. \cite[Proposition 2]{mcduffsegal}): \emph{Let $A$ be a system of local coefficients on $\Omega B M$. Assume that for each $m\in M$ left-multiplication by $m$ on $M_\infty$ is an isomorphism in homology with coefficients in $f^\ast A$. Then $f_\ast\co H_\ast(M_\infty;f^\ast A)\to H_\ast(\Omega BM;A)$ is an isomorphism.} A complete account of this criterion was given by Miller and Palmer \cite[Theorem 4.1]{millerpalmer}. Randal-Williams showed that if $M$ is homotopy commutative, then the assumption is satisfied for any local coefficient system $A$ on $\Omega B M$, and therefore $f\co M_\infty\to \Omega BM$ is acyclic. In the present paper, we use the McDuff-Segal criterion only with trivial integer coefficients, where the assumption that each $m$ acts invertibly on $H_\ast(M_\infty;\mathbb{Z})$ is an immediate consequence of the fact that the homology of $M_\infty$ is $\pi_0(M)$-local (which uses the fact that $\pi_0(M)$ is a right denominator set in $H_\ast(M)$). The upgrade of the integer homology equivalence $f\co M_\infty\to \Omega BM$ to an acyclic map is what our covering space argument achieves, and it uses again the condition that $\pi_0(M)$ is a right denominator set in $H_\ast(M)$.

The following application of Proposition \ref{prop:main} is given in Section \ref{sec:application}.

\begin{proposition} \label{prop:application}
Let $X=\bigsqcup_{n\geq 0} X_n$ be the underlying topological monoid of a commutative monoid in symmetric sequences. Assume that $X\neq X_0$ and that
\begin{enumerate}
\item[(1)] a cofinal right stabilisation sequence in $\pi_0(X)$ exists;
\item[(2)] the natural action\footnote{This action is explained in Section \ref{sec:application}.} of $\Sigma_\infty$ on $H_\ast(X_\infty)$ is trivial.
\end{enumerate}
Then each component of $X_\infty$ has a quasi-perfect fundamental group, the Quillen plus construction gives $X_\infty^+ \simeq \Omega B X$, and the loop space structure on $\Omega B X$ extends to an infinite loop space structure.
\end{proposition}

Proposition \ref{prop:application} shows that in \cite[Theorem 3.1]{AGLT} the assumption on perfection and the assumption that $X_\infty^+$ is an abelian space are not necessary, cf. \cite[Remark 3.2]{AGLT}.

\emph{Notation:} All homology groups in this paper are understood with constant integer coefficients.

\section{Proof of Proposition \ref{prop:main}}

We will work in a homotopy theoretic setting, ie instead of topological monoids we speak of $E_1$-monoids in spaces (or anima), limits and colimits are actually homotopy limits and homotopy colimits, respectively, and commutative diagrams of spaces or monoids commute up to specified homotopy.

We fix an $E_1$-monoid $M$ satisfying the assumptions of Proposition \ref{prop:main}. We choose a cofinal right stabilisation sequence $(s_n)_{n\geq 1}$ in $\pi_0(M)$, and denote by $M_\infty$ the telescope formed using this sequence.

The first step of the proof is a standard spectral sequence argument: By (\ref{eq:mcduffsegal}), the map $f\co M_\infty\to \Omega BM$ is a homology isomorphism with constant integer coefficients. In particular, it is a bijection on $\pi_0$. It suffices to show that for every connected component $C \subseteq M_\infty$ and corresponding component $D\subseteq \Omega BM$ the restricted map $f|_C\co C\to D$ is acyclic.

Let $\tilde{D}$ be the universal covering of $D$ and consider the pullback
\[
\xymatrix{
B \ar[r] \ar[d] & \tilde{D} \ar[d] \\
C \ar[r]^-{f|_C} & D\, .
}
\]
Since $f|_C$ induces abelianisation on fundamental groups, the map $B\to C$ is the universal abelian covering of $C$, ie the connected covering  corresponding to the commutator subgroup of $\pi_1(C)$. Since the horizontal fibres of the square are equivalent, it suffices to show that $B\to \tilde{D}$ is acyclic; but $\tilde{D}$ is simply connected, and so this map is acyclic if and only if it is an integer homology equivalence. This in turn will follow from the Zeeman comparison theorem once we know that the action of the fundamental group on the homology of the coverings is trivial for both $B\to C$ and $\tilde{D}\to D$. For the latter this is clear, because $D$ is a component of a loop space. Therefore, it suffices to show:

\begin{lemma} \label{lem:trivialaction}
For each connected component of $M_\infty$, its fundamental group acts trivially on the homology of its universal abelian covering.
\end{lemma}

The proof of this lemma will occupy the rest of the section.

We begin by assembling the universal abelian coverings of the connected components of $M$ into an $E_1$-monoid over $M$. In the category of spaces $\mathrm{Spc}$, let $\mathrm{Spc}_{\leq 1}^{\mathrm{ab}}\subseteq \mathrm{Spc}$ be the full subcategory of $1$-truncated spaces all of whose connected components have abelian fundamental group. The inclusion $i\co \mathrm{Spc}_{\leq 1}^{\mathrm{ab}}\hookrightarrow \mathrm{Spc}$ admits a left-adjoint $\tau\co \mathrm{Spc} \to \mathrm{Spc}_{\leq 1}^{\mathrm{ab}}$. On an object $X\in \mathrm{Spc}$ with connected components $X_c\subseteq X$, $c\in \pi_0(X)$ we have $\tau(X)\simeq \bigsqcup_{c\in \pi_0(X)} BH_1(X_c)$.

Since $i\tau$ commutes with finite products, $i\tau(M)$ is canonically an $E_1$-monoid, and the unit of the adjunction supplies a map of $E_1$-monoids
\[
\pi\co M\to i\tau(M)\, .
\]
This map is a bijection on $\pi_0$, and we use it to identify $\pi_0(M)$ with $\pi_0(i\tau(M))$. Over each component of $i\tau(M)$ the fibre of $\pi$ is the universal abelian covering of the corresponding component of $M$.

Choose a section (of sets) of the projection $i\tau(M)\to \pi_0(M)$, and let
\[
\sigma \co \mathrm{Free}_{E_1}(\pi_0M)\to i\tau(M)
\]
be its extension over the free $E_1$-monoid generated by $\pi_0(M)$. Note that $\mathrm{Free}_{E_1}(\pi_0M)$ is discrete, and its points are words in the elements of $\pi_0(M)$. For a word $w\in \mathrm{Free}_{E_1}(\pi_0M)$ let $\bar{w}$ denote the element it represents in $\pi_0(M)$. Let $A$ denote the $E_1$-monoid defined by the pullback
\begin{equation} \label{dgr:a}
\xymatrix{
A\ar[r] \ar[d] & M \ar[d]^-{\pi} \\
\mathrm{Free}_{E_1}(\pi_0M) \ar[r]^-{\sigma} & i\tau(M)\, .
}
\end{equation}
The underlying space of $A$ is $\bigsqcup_{w\in \mathrm{Free}_{E_1}(\pi_0M)} A_w$, where $A_w \simeq \mathrm{fib}_{\sigma(w)}(\pi)$ is the universal abelian cover of the connected component $M_{\bar{w}} \subseteq M$.

Next, we show that the action by Deck transformations on each covering $A_w\to M_{\bar{w}}$ is compatible with the monoid structure on $A$. Define an $E_1$-monoid $G$ over $\mathrm{Free}_{E_1}(\pi_0M)$ by the pullback
\begin{equation} \label{dgr:g}
\xymatrix{
G\ar[r] \ar[d] & \mathrm{Free}_{E_1}(\pi_0M)\ar[d]^-{\sigma} \\
\mathrm{Free}_{E_1}(\pi_0M) \ar[r]^-{\sigma} & i\tau(M)\, .
}
\end{equation}
The underlying space of $G$ is $\bigsqcup_{v,w\in \mathrm{Free}_{E_1}(\pi_0M)} G_{v,w}$ where $G_{v,w}$ is the path space in $i\tau(M)_{\bar{v}}$ from $\sigma(v)$ to $\sigma(w)$ -- it is empty unless $\bar{v}=\bar{w}$ in $\pi_0(M)$. In particular, when $v=w$, then $G_{v,v}\simeq \Omega_{\sigma(v)}(i\tau(M)_{\bar{v}})$. Thus, for each $v\in \mathrm{Free}_{E_1}(\pi_0M)$, there is an equivalence of $E_1$-groups
\[
G_{v,v}\simeq H_1(M_{\bar{v}})\, ,
\]
and this is the Deck group for the covering $A_v\to M_{\bar{v}}$. For ease of notation, we will write $G_v$ instead of $G_{v,v}$.

Pulling back diagram (\ref{dgr:a}) along the map $A\to M$, and diagram (\ref{dgr:g}) along the map $A \to \mathrm{Free}_{E_1}(\pi_0M)$, and noting that the two composites from $A$ to $i\tau(M)$ are homotopic, we obtain an equivalence of $E_1$-monoids
\[
A\times_M A\simeq G\times_{\mathrm{Free}_{E_1}(\pi_0M)} A\, .
\]
The left hand side has two maps to $A$ which on the right hand side correspond to the projection onto $A$, and an action of $G$ on $A$, respectively. The latter has component maps $G_{(v,w)} \times A_{v} \to A_{w}$ which for $v=w$ encode the action of the Deck group $G_{v} \simeq H_1(M_{\bar{v}})$ on the covering space $A_v$.

Consider now the following three operations:
\begin{enumerate}
\item[(1)] The action induced by Deck transformations on homology:
\[
\bullet\co \mathbb{Z}[G_v]\otimes H_\ast(A_v)\to H_\ast(A_v)\, .
\]
\item[(2)] The monoid structure of $G$ over $\mathrm{Free}_{E_1}(\pi_0 M)$ provides pairings
\[
\circ\co G_{v}\times G_{w} \to G_{vw}
\]
given by pointwise multiplication of loops in the monoid $i\tau(M)$.
\item[(3)] The Pontryagin product of $A$:
\[
\ast_A\co H_\ast(A_v)\otimes H_\ast(A_w)\to H_\ast(A_{vw})\, .
\]
\end{enumerate}

Because the $G$-action on $A$ is a map of $E_1$-monoids
\[
G\times_{\mathrm{Free}_{E_1}(\pi_0M)} A\to A\,,
\]
the following diagram commutes:
\[
\xymatrix{
\mathbb{Z}[G_v]\otimes H_\ast(A_v) \otimes \mathbb{Z}[G_w]\otimes H_\ast(A_w) \ar[r]^-{\textnormal{swap}} \ar[dd]^-{\bullet\otimes \bullet}& \mathbb{Z}[G_v\times G_w]\otimes H_\ast(A_v)\otimes H_\ast(A_w) \ar[d]^-{\circ \otimes \ast_A} \\ & \mathbb{Z}[G_{vw}]\otimes H_\ast(A_{vw}) \ar[d]^-{\bullet} \\
H_\ast(A_v) \otimes H_\ast(A_w) \ar[r]^-{\ast_A} & H_\ast(A_{vw})
}
\]
This yields the following lemma.

\begin{lemma} \label{lem:relation}
For all $v,w\in \mathrm{Free}_{E_1}(\pi_0M)$, all $g\in G_v$, $h\in G_w$ and all $\alpha\in H_\ast(A_v)$, $\beta\in H_\ast(A_w)$ the following equality holds:
\[
(g\bullet \alpha)\ast_A (h\bullet \beta) = (g\circ h)\bullet (\alpha\ast_A \beta)\, .
\]
\end{lemma}

Recall that the telescope $M_\infty$ is formed along a sequence $(m_n)_{n\geq 1}$ of elements of $M$ which lifts the sequence $(s_n)_{n\geq 1}$ in $\pi_0(M)$. Choosing a lift $\tilde{m}_n \in A_{s_n}$ of $m_n \in M_{s_n}$ for each $n\geq 1$, we obtain a commutative diagram
\[
\xymatrix{
A_x \ar[r]^-{\cdot \tilde{m}_{k+1}}\ar[d] & A_{xs_{k+1}}\ar[r]^-{\cdot \tilde{m}_{k+2}} \ar[d] & A_{xs_{k+1}s_{k+2}} \ar[r] \ar[d] & \cdots \\
M_x \ar[r]^-{\cdot m_{k+1}} & M_{xs_{k+1}}\ar[r]^-{\cdot m_{k+2}} & M_{xs_{k+1}s_{k+2}} \ar[r] & \cdots
}
\]
where each vertical map is the universal abelian covering. On $A$ the subscript $xs_{k+1}\cdots s_n$ means the word in the letters $x,s_{k+1},\dots,s_n$, while on $M$ it means the corresponding element of $\pi_0(M)$. Let $A_{\infty,x,k}$ be the colimit of the top row, and let $M_{\infty,x,k}$ be the colimit of the bottom row. Then $M_{\infty,x,k}$ is what a typical component of $M_\infty$ looks like (so $M_{\infty,x,k}$ corresponds to what we called $C$ at the beginning of this section).

Since sequential colimits preserve fibre sequences, and the truncation $i\tau$ commutes with sequential colimits as well, the induced map
\begin{equation} \label{eq:fibrecolimit}
A_{\infty,x,k} \to M_{\infty,x,k} \to i\tau(M_{\infty,x,k})
\end{equation}
is again a fibre sequence, ie $A_{\infty,x,k}\to M_{\infty,x,k}$ is the universal abelian covering. Moreover, the Deck action for (\ref{eq:fibrecolimit}) is the colimit of the Deck action at the finite stages (apply Lemma \ref{lem:relation} with $h$ the neutral element and $\beta=[\tilde{m}_n]\in H_0(A_{s_n})$ to see that the Deck action is compatible with stabilisation). Thus, to show that this action is trivial in homology, it suffices to show that the action on homology at finite stages is \emph{stably} trivial. This will be shown to hold in Lemma \ref{lem:stablytrivial} below.

 The next lemma is a consequence of our assumption that $\pi_0(M)$ is a right denominator set in $H_\ast(M)$ (which is the condition needed for the ring of right fractions to exist). For convenience, we recall the definition from \cite[(10.3)-(10.5)]{Lam}.

\begin{definition} \label{def:rightdenominatorset}
Let $R$ be a ring. A multiplicative subset $S\subseteq R$ is a \emph{right denominator set} if it satisfies the following two conditions:
\begin{enumerate}
\item[(1)] for all $a\in R$ and $s\in S$ there exist $b\in R$ and $t\in S$ such that $at=sb$;
\item[(2)] for all $a\in R$, if $sa=0$ for some $s\in S$, then there exists $t\in S$ such that $at=0$.
\end{enumerate}
\end{definition}

Let $\ast$ denote the Pontryagin product in $H_\ast(M)$.

\begin{lemma} \label{lem:centrality}
Fix $x\in \pi_0(M)$ and $k\geq 0$. Then, for every $n\geq k$ and every $\alpha\in H_\ast(M_{xs_{k+1}\cdots s_n})$ there exist $N> n$ and $\beta \in H_\ast(M_{s_{n+1}\cdots s_N})$ such that the following equality holds:
\[
\alpha \ast (s_{n+1}\cdots s_N)=(x s_{k+1}\cdots s_n)\ast \beta\, .
\]
\end{lemma}
\begin{proof}
Condition (1) of Definition \ref{def:rightdenominatorset} applied to $a=\alpha$ and $s=xs_{k+1}\cdots s_n$ provides an element $t\in \pi_0(M)$ and an element $\gamma \in H_\ast(M)$ such that
\[
\alpha \ast t=(xs_{k+1}\cdots s_n)\ast \gamma\, .
\]
Without loss of generality, $\gamma=\gamma_1+\cdots +\gamma_m$ with $\gamma_j\in H_\ast(M_{c_j})$ and
\[
xs_{k+1}\cdots s_n t=xs_{k+1}\cdots s_n c_j
\]
in $\pi_0(M)$, for every $j=1,\dots,m$. For each $j=1,\dots,m$, condition (2) of Definition \ref{def:rightdenominatorset} yields  an element $u_j\in \pi_0(M)$ such that $(t-c_j)u_j=0$. By cofinality of the sequence $(s_n)_{n\geq 1}$, there is a $u\in \pi_0(M)$ such that each $u_j$ is a left factor of $u$, and hence $(t-c_j)u=0$ for all $j$.

Now choose $N > n$ and $v\in \pi_0(M)$ such that $tuv=s_{n+1}\cdots s_N$. Set
\[
\beta:=\gamma\ast (uv)\, .
\]
Then $\beta$ lies in $H_\ast(M_{s_{n+1}\cdots s_N})$, and it does the job.
\end{proof}

We will only need the following corollary.

\begin{lemma} \label{lem:centralitygroup}
 Fix $x\in \pi_0(M)$ and $k\geq 0$. Let $xs_{k+1}\cdots s_n$ be the word with letters $x,s_{k+1},\dots ,s_n$. Then, for every $n\geq k$ and every $g\in G_{xs_{k+1}\cdots s_n}$ there exist $N> n$ and $h\in G_{s_{n+1}\cdots s_N}$ (where $s_{n+1}\cdots s_N$ is the word with letters $s_{n+1},\dots ,s_N$) such that the following equality holds:
\[
g\circ e_{s_{n+1}\cdots s_N}=e_{xs_{k+1}\cdots s_n}\circ h
\]
where $e_w\in G_w$ denotes the neutral element.
\end{lemma}
\begin{proof}
This follows by applying Lemma \ref{lem:centrality} in homology degree one, using the identification $G_v\simeq H_1(M_{\bar{v}})$ under which left- or right-multiplication by $e_w$ become left- or right-multiplication, respectively, by $\bar{w}\in \pi_0(M)\subseteq H_0(M)$ in the Pontryagin ring of $M$.
\end{proof}

The next lemma shows that the action by Deck transformations on the homology of $A$ is trivial after right-stabilisation. It implies Lemma \ref{lem:trivialaction} and finishes the proof of Proposition \ref{prop:main}.

\begin{lemma} \label{lem:stablytrivial}
Fix $x\in \pi_0(M)$ and $k\geq 0$. Then, for every $n\geq k$, every $g\in G_{xs_{k+1}\cdots s_n}$ and every $\alpha \in H_\ast(A_{xs_{k+1}\cdots s_n})$ there is $N> n$ such that
\[
(g\bullet \alpha) \ast_A [\tilde{m}_{n+1}\cdots \tilde{m}_{N}] =  \alpha \ast_A [\tilde{m}_{n+1}\cdots \tilde{m}_{N}]
\] 
where $[\tilde{m}_{n+1}\cdots \tilde{m}_{N}]\in H_0(A_{s_{n+1}\cdots s_N})$ is the class of $\tilde{m}_{n+1}\cdots \tilde{m}_{N}$.
\end{lemma}
\begin{proof}
By Lemma \ref{lem:centralitygroup} there exists an $N > n$ and $h\in G_{s_{n+1}\cdots s_N}$ such that
\[
g\circ e_{s_{n+1}\cdots s_N} = e_{xs_{k+1}\cdots s_n}\circ h\, .
\]
Since $e_{s_{n+1}\cdots s_N}$ is the neutral element of the Deck group,
\[
(g\bullet \alpha) \ast_A [\tilde{m}_{n+1}\cdots \tilde{m}_{N}] =(g\bullet \alpha) \ast_A (e_{s_{n+1}\cdots s_N}\bullet [\tilde{m}_{n+1}\cdots \tilde{m}_{N}]) 
\]
which by Lemma \ref{lem:relation} can be written
\[
(g\circ e_{s_{n+1}\cdots s_N})\bullet (\alpha\ast_A [\tilde{m}_{n+1}\cdots \tilde{m}_{N}])\, .
\]
Applying Lemma \ref{lem:centralitygroup} and once again Lemma \ref{lem:relation}, this equals
\[
(e_{xs_{k+1}\cdots s_n}\bullet \alpha) \ast_A (h\bullet [\tilde{m}_{n+1}\cdots \tilde{m}_{N}])\, .
\]
Since $A_{s_{n+1}\cdots s_N}$ is connected, the action of $h$ on $H_0(A_{s_{n+1}\cdots s_N})$ is trivial. Hence, the last expression simplifies to $\alpha\ast_A [\tilde{m}_{n+1}\cdots \tilde{m}_N]$, as wanted.
\end{proof}

\section{Proof of Proposition \ref{prop:application}} \label{sec:application}

\begin{definition} \label{def:symmetricsequence}
A \emph{commutative monoid in symmetric sequences} is a sequence of simplicial sets $(X_n)_{n\geq 0}$ equipped with the following data:
\begin{enumerate}
\item[(1)] a left $\Sigma_n$-action on each $X_n$;
\item[(2)] for all $m,n\geq 0$ a strictly associative product
\[
\mu_{m,n}\co X_m\times X_n\to X_{m+n}
\]
which is equivariant for the block inclusion $\Sigma_m\times \Sigma_n\subseteq \Sigma_{m+n}$;
\item[(3)] a basepoint $\ast\in X_0$ acting as a strict two-sided unit;
\end{enumerate}
and satisfying the following condition: For all $m,n\geq 0$ the diagram
\[
\xymatrix{
X_m\times X_n\ar[r]^-{\mu_{m,n}} \ar[d]^-{\textnormal{swap}} & X_{m+n}\ar[d]^-{\tau_{m,n}\cdot } \\
X_{n}\times X_m \ar[r]^-{\mu_{n,m}} & X_{n+m}
}
\]
commutes, where $\tau_{m,n}\in \Sigma_{m+n}$ is the permutation transposing blocks of $m$ and $n$ letters.
\end{definition}

For convenience, we shall assume that $X_1$ is non-empty (hence, $X_n\neq \emptyset$ for all $n\geq 0$). See Remark \ref{rem:x1empty}.

By forgetting the symmetric group actions, a commutative monoid in symmetric sequences has an underlying $E_1$-monoid $X=\bigsqcup_{n\geq 0} X_n$.

Now let such an $X$ be given, and assume that $(s_n)_{n\geq 1}$ is a cofinal right stabilisation sequence in $\pi_0(X)$, and let $X_\infty$ be the mapping telescope formed with respect to a sequence of representatives $(x_n)_{n\geq 1}$ for $(s_n)_{n\geq 1}$. We claim that the symmetric group $\Sigma_\infty$ acts naturally on $H_\ast(X_\infty)$. For each $N\geq 0$ let $X_{\geq N}:=\bigsqcup_{n\geq N} X_n$. The symmetric group $\Sigma_N$ acts on $X_{\geq N}$ by
\[
\sigma\cdot x:=(\sigma\oplus 1_{n-N})\cdot x
\]
for $\sigma\in \Sigma_N$ and $x\in X_n$ with $n\geq N$. Here, $\sigma\oplus 1_{n-N}\in \Sigma_n$ is the permutation that acts by $\sigma$ on the first $N$ letters and by the identity on the last $n-N$ letters. It follows from item (2) of Definition \ref{def:symmetricsequence} that this action is equivariant with respect to the right action $X_{\geq N}\times X\to X_{\geq N}$.

The inclusion $X_{\geq N}\hookrightarrow X$ gives a map of colimit systems
\[
\xymatrix{
X_{\geq N}\ar[r]^-{\cdot x_1} \ar[d]&X_{\geq N} \ar[r]^-{\cdot x_2} \ar[d]&X_{\geq N} \ar[r]^-{\cdot x_3} \ar[d]& \cdots\\
X \ar[r]^-{\cdot x_1} & X \ar[r]^-{\cdot x_2} & X \ar[r]^-{\cdot x_3} & \cdots
}
\]
Let $(X_{\geq N})_\infty$ denote the telescope of the upper sequence. By what we just said, the symmetric group $\Sigma_N$ acts on each $X_{\geq N}$ compatibly with the stabilisation maps, and hence it acts on $H_\ast((X_{\geq N})_\infty)$.

Because $X_1\neq \emptyset$, and every element of $\pi_0(X)$ is a left factor of a product $s_k\cdots s_{k+n}$, the degrees of the products $s_1\cdots s_n$, $n\geq 1$ are unbounded (with respect to the $\mathbb{N}$-grading on $\pi_0(X)$). It follows that any component $X_n\subseteq X$ is pushed in the colimit system into arbitrarily high degrees, and hence by cofinality the induced map of telescopes is an equivalence:
\[
(X_{\geq N})_\infty \xrightarrow{\simeq} X_\infty\, .
\]
In particular, each $\Sigma_N$ acts on $H_\ast(X_\infty)$ via the isomorphism induced by this equivalence.

Under the isomorphism $
H_\ast((X_{\geq N+1})_\infty) \cong H_\ast((X_{\geq N})_\infty)$, induced by the inclusion $X_{\geq N+1}\hookrightarrow X_{\geq N}$, the action of $\sigma\in \Sigma_N$ on $H_\ast((X_{\geq N})_\infty)$ and the action of $\sigma\oplus 1_1$ on $H_\ast((X_{\geq N+1})_\infty)$ agree. Therefore, the $\Sigma_N$-actions on $H_\ast(X_\infty)$ for various $N$ assemble into an action of $\Sigma_\infty$ on $H_\ast(X_\infty)$.

\begin{lemma} \label{lem:denominatorset}
In the setting above, suppose that the action of $\Sigma_\infty$ on $H_\ast(X_\infty)$ is trivial. Then $\pi_0(X)$ is a right denominator set in $H_\ast(X)$.
\end{lemma}
\begin{proof}
We begin by checking condition (1) of Definition \ref{def:rightdenominatorset}. Let $\alpha \in H_\ast(X)$ and $s\in \pi_0(X_m)$. We must find a $t\in \pi_0(X)$ and a $\beta\in H_\ast(X)$ such that $\alpha\ast t=s\ast \beta$, where $\ast$ is the Pontryagin product in $H_\ast(X)$. Write $\alpha = \alpha_1+\cdots +\alpha_r$ with $\alpha_j\in H_\ast(X_{n_j})$, $n_j\geq 0$. In $H_\ast(X)$ we have, by Definition \ref{def:symmetricsequence}, for each $j=1,\dots,r$ an equality
\[
\tau_{n_{j},m}\cdot (\alpha_j \ast s) = s\ast \alpha_j\, .
\]
Since $\Sigma_\infty$ acts trivially on $H_\ast(X_\infty)$, there is a $t_j\in \pi_0(X)$ such that
\[
(s\ast \alpha_j) \ast t_j =(\tau_{n_j,m}\cdot (\alpha_j\ast s))\ast t_j = (\alpha_j\ast s)\ast t_j\, .
\]
By cofinality of the sequence $(s_n)_{n\geq 1}$, there is $N\geq 1$ such that, for every $j=1,\dots,r$, $t_j$ is a left factor of $s_1\cdots s_N$ in $\pi_0(X)$. Using $s_1\cdots s_N$ instead of $t_j$ and summing over all $j$ yields
\[
(s\ast \alpha) \ast s_1\cdots s_N = (\alpha\ast s)\ast s_1\cdots s_N\,,
\]
which is condition (1) with $t=ss_1\cdots s_N$ and $\beta=\alpha\ast(s_1\cdots s_N)$.

Condition (2) of Definition \ref{def:rightdenominatorset} follows more easily and does not need the assumption that $\Sigma_\infty$ acts trivially: Assume that $s\ast \alpha =0$, where $s\in \pi_0(X_m)$ and $\alpha\in H_\ast(X)$. Without loss of generality, $\alpha\in H_\ast(X_n)$ for some $n\geq 0$. Then $\alpha\ast s=\tau_{m,n}\cdot (s\ast \alpha)=0$.
\end{proof}

By Lemma \ref{lem:denominatorset}, we can apply Proposition \ref{prop:main} to $X$, showing that the map $X_\infty\to \Omega B X$ is acyclic. A standard consequence (see \cite[Section 1.1]{oscar}) is that the fundamental group of each component of $X_\infty$ has a perfect commutator subgroup and Quillen's plus construction $X_\infty^+$ is weakly equivalent to $\Omega B X$. This proves the first part of Proposition \ref{prop:application}.

The proof of the proposition is finished by an adaptation of the argument of \cite[Theorem 3.1]{AGLT}, showing that the $E_1$-structure on $\Omega B X$ extends to an $E_\infty$-structure: Since $(X_n)_{n\geq 0}$ is a commutative monoid in symmetric sequences, the $\Sigma_n$-homotopy orbits of $X_n$ assemble into an $E_\infty$-monoid
\[
M:=\bigsqcup_{n\geq 0} (X_n)_{h\Sigma_n}\,,
\]
see \cite[Section 2.4]{AGLT}. Let $P:=\bigsqcup_{n\geq 0} B\Sigma_n$. Then the canonical map $p\co M\to P$ induced by the projections $(X_n)_{h\Sigma_n}\to B\Sigma_n$ is a map of $E_\infty$-monoids. Consider the pullback of $E_1$-monoids
\begin{equation} \label{eq:pullbacke1monoid}
\xymatrix{
X \ar[r] \ar[d] & M \ar[d]^-{p} \\
\mathbb{N} \ar[r]^-{j} & P
}
\end{equation}
where the map $j\co \mathbb{N}\to P$ is induced by sending $1\in \mathbb{N}$ to a point in $B\Sigma_1$ (note that $\mathbb{N}$ is the free $E_1$-monoid on a single generator). In general, group-completion does not preserve pullback squares, but in our situation it does:

\begin{lemma} \label{lem:pullbacke1}
The above square induces a pullback square of $E_1$-groups
\begin{equation} \label{eq:pullbacke1}
\xymatrix{
\Omega B X \ar[r] \ar[d] & \Omega B M \ar[d]^-{\Omega B (p)} \\
\mathbb{Z} \ar[r]^-{\Omega B (j)} & \Omega B P\, .
}
\end{equation}
\end{lemma}
\begin{proof}
The sequence $(s_n)_{n\geq 1}$ in $X$ maps to a cofinal right stabilisation sequence for each of the monoids $\mathbb{N}$, $M$ and $P$ (note that $\pi_0(X)\to \mathbb{N}$ is surjective, because $X_1\neq \emptyset$). Forming telescopes in (\ref{eq:pullbacke1monoid}), and using the fact that sequential colimits preserve pullbacks in spaces, we obtain a pullback square in spaces
\[
\xymatrix{
X_\infty \ar[r] \ar[d] & M_\infty \ar[d] \\
\mathbb{Z} \ar[r] & \mathbb{Z}\times B\Sigma_\infty.
}
\]
There are natural equivalences
\[
X_\infty^+\simeq \Omega B X,\quad M_\infty^+\simeq \Omega B M\,,\quad \mathbb{Z}\times B\Sigma_\infty^+\simeq \Omega B P\, .
\]
In each case, the plus construction is applied to each connected component and with respect to the commutator subgroup of the fundamental group. For $X$ the equivalence follows from what was said before the lemma, and for $M$ and $P$ we can apply Proposition \ref{prop:main} (in fact, since $M$ and $P$ are homotopy commutative, \cite[Theorem 1.1]{oscar} suffices). Thus, to prove the lemma, it suffices to show that the pullback square remains a pullback square after plus construction.

For $k\in \mathbb{Z}$, let $X_{\infty,k}$ denote the collection of components of $X_\infty$ over $k$, which may be disconnected, and let $M_{\infty,k}$ be the collection of components of $M_\infty$ over $\{k\}\times B\Sigma_\infty$. There is a fibre sequence
\[
X_{\infty,k} \to M_{\infty,k} \to \{k\}\times B\Sigma_\infty
\]
whose monodromy induces the given action of $\Sigma_\infty$ on the homology of $X_{\infty,k}$. Since $\Sigma_\infty$ is assumed to act trivially on $H_0(X_\infty)\cong \mathbb{Z}[\pi_0(X_\infty)]$, each connected component of $X_\infty$ is fixed by the $\Sigma_\infty$-action. Therefore, $X_{\infty,k}\to M_{\infty,k}$ is a bijection on $\pi_0$. Thus, for each connected component $C\subseteq X_{\infty,k}$, and corresponding component $M_C\subseteq M_{\infty,k}$, we obtain a fibre sequence
\[
C\to M_C \to B\Sigma_\infty\, .
\]
Apply to each term the plus-construction with respect to the maximal perfect subgroup of the fundamental group, which in each case is the commutator subgroup. By a theorem of Berrick \cite[Theorem 1.1]{Berrick}, the sequence
\[
C^+\to M_C^+ \to B\Sigma_\infty^+
\]
is still a fibre sequence. This uses the fact that $C^+$ is equivalent to a component of a loop space (which follows from $X_\infty^+\simeq \Omega B X$), and therefore nilpotent, and that $\Sigma_\infty$ acts trivially on $H_\ast(C)$. This finishes the proof.
\end{proof}

Let $\mathrm{fib}_0(\Omega B(p))$ denote the fibre of $\Omega B(p)$ over the basepoint of $\Omega B P$.

\begin{lemma} \label{lem:productinfiniteloopstructure}
There is an equivalence of $E_1$-groups
\[
\mathbb{Z} \times \mathrm{fib}_0(\Omega B (p)) \xrightarrow{\simeq} \Omega B X\, .
\]
\end{lemma}
\begin{proof}
Since $\mathbb{Z}$ is the free $E_1$-group on a single generator, the left hand vertical map in diagram (\ref{eq:pullbacke1}) has an $E_1$-section $\mathbb{Z}\to \Omega B X$. Let  $\sigma\co \mathbb{Z}\to \Omega B M$ be its composition with the obvious map $\Omega B X\to \Omega B M$, and let $i\co \mathrm{fib}_0(\Omega B (p))\to \Omega B M$ be the inclusion of the fibre. Multiplication in the $E_\infty$-monoid $\Omega B M$ is an $E_1$-map, and so the composite
\[
\mathbb{Z}\times \mathrm{fib}_0(\Omega B(p))\xrightarrow{\sigma\times i} \Omega B M\times \Omega B M\xrightarrow{\textnormal{mult.}} \Omega B M
\]
is an $E_1$-map. By construction its composite with $\Omega B (p)$ is homotopic, as an $E_1$-map, to the composite
\[
\mathbb{Z}\times \mathrm{fib}_0(\Omega B(p)) \xrightarrow{\textnormal{proj.}} \mathbb{Z} \xrightarrow{\Omega B(j)} \Omega B P\,.
\]
From the pullback in Lemma \ref{lem:pullbacke1}, we therefore obtain an $E_1$-map over $\mathbb{Z}$
\[
\mathbb{Z} \times \mathrm{fib}_0(\Omega B (p)) \to \Omega B X\,.
\]
In the fibre over $k\in \mathbb{Z}$, this map is left-translation by $\sigma(k)$ in the infinite loop space $\Omega BM$, which is an equivalence. Hence, the displayed map is an equivalence as well.
\end{proof}

Since $\Omega B (p)$ is a map of infinite loop spaces, $\mathrm{fib}_0(\Omega B (p))$ is an infinite loop space. Thus, the $E_1$-structure on $\Omega B X$ can be extended to the product infinite loop space structure. This finishes the proof of Proposition \ref{prop:application}.

\begin{remark} \label{rem:x1empty}
Definition \ref{def:symmetricsequence} implies that $X_0$ is a strictly abelian monoid. So, in the case that $X=X_0$ the conclusion of Proposition \ref{prop:application} is elementary. If $X\neq X_0$ but $X_1=\emptyset$, the same proof as the one given shows that $\Omega B X$ is equivalent, as an $E_1$-group, with $A\times \mathrm{fib}_0(\Omega B(p))$, where $A$ is the Grothendieck group of $\mathrm{im}(\pi_0(X)\to \mathbb{N})$ -- which is also $\mathbb{Z}$.
\end{remark}

\bibliographystyle{amsalpha}
\bibliography{refs}

\end{document}